\documentclass[preprint,12pt]{elsarticle}

\usepackage{graphicx}
\usepackage{amssymb, amsthm, amsmath, fullpage, hyperref}
\usepackage{graphicx, enumerate}
\usepackage{tikz-network}
\usetikzlibrary{positioning, quotes}
\usepackage[utf8]{inputenc}
\usepackage{cleveref}
\crefname{thm}{Theorem}{Theorems}
\crefname{lem}{Lemma}{Lemmas}
\usepackage{amsfonts}
\usepackage{amssymb}
\usepackage{MnSymbol}
\usepackage{amsmath,amsthm,dsfont}
\usepackage[english]{babel}
\usepackage{makeidx}
\usepackage{mathtools}
\usepackage{bbm}
\usepackage{fullpage,amsfonts,amssymb,epsfig,epstopdf,amsmath,url,array,titlesec}
\usepackage{mathtools}
\usepackage{tikz}
\usepackage{float}
\usepackage{makeidx}
\usepackage[mathscr]{euscript}
\usepackage{nomencl}
\usepackage{comment}
\allowdisplaybreaks

\theoremstyle{plain}
\newtheorem{thm}{Theorem}[section]

\theoremstyle{definition}
\newtheorem{defn}[thm]{Definition}

\newtheorem{rem}[thm]{Remark}

\newtheorem{claim}[thm]{Claim}

\newtheorem{theorem}{Theorem}
\newtheorem{case}{Case}
\makeatletter
\@addtoreset{case}{subsection}

\newtheorem{subcase}{Subcase}[case]

\makeatother
\newtheorem{subsubcase}{Subsubcase}[subcase]

\journal{Discrete Mathematics}

\begin{document}

\begin{frontmatter}

\title{Off-diagonal Rado number for $x+y+c=z$ and $x+qy=z$}

\author[1]{Rajat Adak}
\ead{rajatadak@iisc.ac.in}
\author[1]{Yash Bakshi}
\ead{yashbakshi@iisc.ac.in}
\author[1]{L. Sunil Chandran\corref{cor1}}
\ead{sunil@iisc.ac.in}
\author[1]{Saraswati Girish Nanoti}
\ead{saraswatig@iisc.ac.in}
\cortext[cor1]{Corresponding author}
\address[1]{Department of Computer Science and Automation, Indian Institute of Science, Bangalore, India}

\begin{abstract}
Ramsey-type problems for linear equations began with Schur’s theorem and were
systematically generalized by Richard Rado.  In the off-diagonal framework for two colors, one considers two different linear equations $(\mathcal{E}_1,\mathcal{E}_2)$ and determines the minimum integer $N$ for which any red–blue coloring of $\{1,2,\dots,N\}$ forces either a red solution of the  equation $\mathcal{E}_1$  or a blue solution of the equation $\mathcal{E}_2$.  
In this work, we study off-diagonal Rado numbers for non-homogeneous linear equations of the forms $x+y+c=z$ and $x+qy=z$. We determine the exact two-color off-diagonal Rado number $R_2(c,q)$ associated with this system of equations.
\end{abstract}


\end{frontmatter}
\section{Introduction}

Schur’s classical theorem \cite{1916Schur} asserts that for any finite coloring of the
positive integers, there exists a monochromatic solution to the equation $x+y=z$.  Rado \cite{rado1933studien} extended this result by characterizing which systems of linear equations are \emph{regular}, meaning that every finite coloring of the integers contains a monochromatic solution.

While most early work focused on diagonal Ramsey-type problems, where the solution to same equation must appear monochromatically, more recent research has explored \emph{off-diagonal} variants.  In these problems, one seeks to force a monochromatic solution to one of two distinct equations, each assigned a different color. This framework constitutes a natural generalization of the classical Schur numbers and expands the scope of combinatorial and structural analysis..

In particular, off-diagonal Rado numbers describe the interaction between two different linear configurations and often require methods that differ from those used in the diagonal case. Even for simple linear equations, determining exact off-diagonal Rado numbers is highly non-trivial and remains open in many cases.

\begin{defn}[Rado number]
Let $\mathcal{E}$ be a linear equation with integer coefficients. 
The \emph{t-color Rado number} of $\mathcal{E}$, denoted by $R_t(\mathcal{E})$, is the smallest positive integer $N$ such that every 
coloring using $t$-colors of the set $[1,N]$ contains a monochromatic solution to $\mathcal{E}$. If no such integer $N$ exists, then $R_t(\mathcal{E})$ is said to be infinite.
\end{defn}

\begin{defn}[Off-diagonal Rado number]
Let $\mathcal{E}_0$ and $\mathcal{E}_1$ be linear equations with integer coefficients. 
The \emph{two-color off-diagonal Rado number} of the pair of equations $(\mathcal{E}_0,\mathcal{E}_1)$, denoted by $R_2(\mathcal{E}_0,\mathcal{E}_1)$, is the smallest positive integer $N$ such that every red--blue coloring of the set $[1,N]$ contains either a red monochromatic solution to $\mathcal{E}_0$, or a blue monochromatic solution to $\mathcal{E}_1$. If no such integer $N$ exists, then $R_2(\mathcal{E}_0,\mathcal{E}_1)$ is said to be infinite.
\end{defn} This definition reduces to the classical Rado number when $\mathcal{E}_0=\mathcal{E}_1$, and to Schur numbers when both equations are $x+y=z$.
\section{Previous Work}

Schur \cite{1916Schur} proved that the equation \(x+y=z\) is regular and introduced what are now called Schur numbers. Rado \cite{rado1933studien} generalized this by characterizing all linear systems that are regular, thereby determining precisely which systems necessarily admit monochromatic solutions under every finite coloring. Burr and Loo \cite{BurrLoo1975} initiated the systematic study of exact Rado numbers for specific linear equations ($x+y+c=z$) and provided bounds and constructions for several families. Beutelspacher and Brestovansky \cite{beutelspacher2006generalized} showed that for any positive integer $k$, the $2$-color Rado number for the equation
$x_1 + x_2 + \cdots + x_k = x_0$ is $k^2 + k - 1$. 

There are many single linear equations which has been studied for 2 or 3 color settings.
\[
x+y+c=kz, \quad (\text{in } \cite{jones2004two})
\]
\[
x_1+x_2+..+x_{m-1} +c= x_m \quad ( c < 0)\quad (\text{in } \cite{kosek2001rado})
\]

There have been some results involving nonlinear
equations as well::
\[
\frac{1}{x} + \frac{1}{y} = \frac{1}{z}, \quad (\text{in } \cite{doss2013excursion})
\]
\[
x + y^{n} = z. \quad (\text{in } \cite{MyersParrish2018})
\]

The \textbf{off-diagonal} setting was first studied by Robertson and Schaal~\cite{robertson_schaal2001}, who determined exact two-color off-diagonal generalized Schur numbers. Specifically, for positive integers \(k\) and \(l\), they determined the smallest integer \(S=S(k,l)\) such that any red--blue coloring of \([1,S]\) yields either a red solution to $x_1+x_2+\cdots+x_k = x_0$ or a blue solution to $x_1+x_2+\cdots+x_l = x_0$.

Myers and Robertson \cite{myers2006two} later expanded this framework and obtained exact formulas for several families of homogeneous off-diagonal Rado numbers, primarily focusing on equations of the form \(x+qy=z\) and related systems.  However, off-diagonal problems involving \emph{non-homogeneous} linear equations, such as shifted Schur-type equations, remain far less explored and present a range of open problems.

\section{Main Result}
\begin{theorem}
Let $c\ge1$ and $q\ge 1$ be positive integers. Then the two-color off-diagonal Rado number $R_2(c,q)$ for the system 

\begin{align*}
x+y+c &= z \quad \text{(red)},\\
x+qy  &= z \quad \text{(blue)},
\end{align*}
is:

\[
R_2(c,q) =
\begin{cases}
\infty, & \text{if $q$ is odd, $c$ is odd} \\
2c+4,  & \text{if $q=1$ and $c$ is even} \\
(q+1)(c+2)+c+1 & \text{otherwise }
\end{cases}
\]
\end{theorem}
\begin{proof} We denote $x+y+c =z$ as the $c$ equation and $x+qy=Z$ as the $q$ equation. We evaluate $R_2(c,q)$ case by case.
\begin{case}
    Both $c$ and $q$ are odd.
\end{case}
We define a $2$-coloring $\chi : \mathbb{N} \longrightarrow \{\text{Red}, \text{Blue}\}$ as follows:
$$
\chi(n) =
\begin{cases}
\text{Blue}, & \text{if } n \text{ is odd}, \\
\text{Red},  & \text{if } n \text{ is even}.
\end{cases}
$$
We show that under this coloring, there is neither a red solution to the $c$ equation, nor a blue solution to the $q$ equation.

\medskip

\noindent
\textbf{Absence of a red solution to $c$ equation:}
 If $x$ and $y$ are both red, then $x$ and $y$ are even, and hence $x + y + c$ is odd. Thus, $z$ is odd and therefore colored blue. Hence, it is impossible for all variables to be red, and no red solution exists for the $c$ equation.

\medskip

\noindent
\textbf{Absence of a blue solution to $q$ equation:}  
 If $x$ and $y$ are both blue, then $x$ and $y$ are odd, and therefore $x + qy$ is even. Thus $z$ is even and hence colored red. Consequently, no blue solution exists for the $q$ equation.

\medskip

\noindent

The above coloring avoids both a red solution to the $c$ equation and a blue solution to $q$ equation. Therefore, there exists a coloring of $\mathbb{N}$ such that no monochromatic solution exists to either equation. This implies that the corresponding $R_2(c,q)$ is infinite.

\begin{case}
    $q=1$ and $c$ is even.
\end{case}
For $q=1$ the system becomes:
\begin{align*}
x+y+c &= z \quad \text{(red)},\\
x+y  &= z \quad \text{(blue).}
\end{align*}
We claim that $R_2(c,q)$ is \(2c+4\). This means that $2c+4$ is the smallest positive integer \(r\) such that every red--blue coloring of the integers \(\{1,2,\dots,r\}\) contains either a red solution to $c$, or a blue solution to $q$.

\begin{claim}\label{clm3.1}
    $R_2(c,q) \geq 2c+4$
\end{claim}
\begin{proof} Let $R = \{1,2,\dots,c+1\}$ and $B = \{c+2,c+3,\dots,2c+3\}$. Consider the 2-coloring of $[1,2c+3]$ defined by
$$
\chi(n) = \text{Red }, \text{ if $n \in R$},
\qquad
\chi(n) = \text{Blue }, \text{ if $n \in B$}.
$$
We will show that under this coloring there is
\emph{no red solution} to equation $c$ and \emph{no blue solution} to equation $q$ with all variables in $[1,2c+3]$.

\smallskip
\noindent\textbf{No red solution to $c$ equation:}\\
If $x$ and $y$ are red, then $1\le x,y\le c+1$, hence $z=x+y+c \ge 1+1+c = c+2$. Therefore $z\in\{c+2,\dots,2c+3\}$, and thus colored blue. So $(x,y,z)$ cannot be a monochromatic red triple, and thus there is no red solution.

\smallskip
\noindent\textbf{No blue solution to $q$ equation:}\\
If $x$ and $y$ are blue, then $c+2\le x,y\le 2c+3$, hence $z=x+y \;\ge\; (c+2)+(c+2) \;=\; 2c+4 \;>\; 2c+3$.
So $z$ does not lie in the interval $[1,2c+3]$, contradicting the requirement that all variables are within the colored set. Hence, there is no blue solution to $x+y=z$ in $[1,2c+3]$.

\smallskip
\noindent
Consequently, this coloring avoids both types of forbidden solutions on $[1,2c+3]$. In particular, the off-diagonal Rado number satisfies $R_2(c,q)\ge 2c+4$.\end{proof}

\begin{claim}\label{clm3.2}
    $R_2(c,q) \leq 2c+4$
\end{claim}
\begin{proof}We show that any red--blue coloring of $[1,2c+4]$ gives either a red solution to $x+y+c=z$ or blue solution to $x+y=z$. Assume for the sake of contradiction, there is a coloring such that no red solution to $x+y+c=z$ and no blue solution to $x+y=z$ exists, respectively. We start with an arbitrary coloring of numbers starting with $1$ (and using numbers upto $2c+4$), and try to extend this coloring such that the desired conditions are satisfied, but end up with either a red solution to the first equation or a blue solution to the second equation. 
\vspace{2mm}

\noindent\textbf{Notation.} We adopt the convention that, when trying to avoid monochromatic (red or blue) solutions, a given coloring together with a particular solution—represented as an ordered triple \((x,y,z)\)
—may force an extension of the coloring. For example if $1$ is red, taking $x = y =1$ in $c$ equation, to avoid a red solution, we get,
\[
(1,1,c+2)_c\;\Longrightarrow\; c+2 \text{ is blue}.
\]
\begin{subcase} Assume that $1$ is colored red. \end{subcase}
\begingroup
\setlength{\jot}{2pt} 
\begin{align*}
(1,1,c+2)_c&\Longrightarrow {c+2} \text{ is blue}.\\
(c+2,c+2,2c+4)_q&\Longrightarrow {2c+4} \text{ is red}.\\
\left(\tfrac{c}{2}+2,\tfrac{c}{2}+2,2c+4\right)_c&\Longrightarrow {\tfrac{c}{2}+2} \text{ is blue}.\\
\left(c+2,\tfrac{c}{2}+2,\tfrac{3c}{2}+4\right)_q&\Longrightarrow {\tfrac{3c}{2}+4} \text{ is red}.\\
\left(1,\tfrac{c}{2}+3,\tfrac{3c}{2}+4\right)_c&\Longrightarrow {\tfrac{c}{2}+3} \text{ is blue}.\\
\left(\tfrac{c}{2}+2,\tfrac{c}{2}+3,c+5\right)_q&\Longrightarrow {c+5} \text{ is red}.\\
(1,4,c+5)_c&\Longrightarrow {4} \text{ is blue}.\\
(2,2,4)_q&\Longrightarrow {2} \text{ is red}.\\
\left(\tfrac{c}{2}+2,\tfrac{c}{2}+2,c+4\right)_q&\Longrightarrow {c+4} \text{ is red}.
\end{align*}
\text{Thus, }$(2,2,c+4)$\text{ is a red solution to the first equation, a contradiction.}
\endgroup

\begin{subcase}
    $1$ is colored blue.
\end{subcase} 
\begingroup

\begin{align*}
(1,1,2)_q&\Longrightarrow {2} \text{ is red}.\\
(2,2,c+4)_c&\Longrightarrow {c+4} \text{ is blue}.\\
\left(\tfrac{c}{2}+2,\tfrac{c}{2}+2,c+4\right)_q&\Longrightarrow {\tfrac{c}{2}+2} \text{ is red}.\\
\left(2,\tfrac{c}{2}+2,\tfrac{3c}{2}+4\right)_c&\Longrightarrow {\tfrac{3c}{2}+4} \text{ is blue}.\\
\left(\tfrac{c}{2},c+4,\tfrac{3c}{2}+4\right)_q&\Longrightarrow {\tfrac{c}{2}} \text{ is red}.\\
\left(\tfrac{c}{2},\tfrac{c}{2},2c\right)_c&\Longrightarrow {2c} \text{ is blue}.\\
(c,c,2c)_q&\Longrightarrow {c} \text{ is red}.\\
(2,c,2c+2)_c&\Longrightarrow {2c+2} \text{ is blue}.\\
(c+1,c+1,2c+2)_q&\Longrightarrow {c+1} \text{ is red}.\\
(2,c+1,2c+3)_c&\Longrightarrow {2c+3} \text{ is blue}.\\
\left(\tfrac{c}{2}+2,\tfrac{c}{2}+2,2c+4\right)_c&\Longrightarrow {2c+4} \text{ is blue}.
\end{align*}
\text{Thus }$(1,2c+3,2c+4)$\text{ is a blue solution to the second equation, a contradiction.}
\endgroup

Thus, we arrive at a contradiction in both the subcases. Therefore, $R_2(c,q) \leq 2c+4$.
\end{proof}
From Claim \ref{clm3.1} and \ref{clm3.2} we get that when $q=1$ and $c$ is even, $R_2(c,q) = 2c+4$.

\begin{rem}
Note that the same proof does not follow when $c=0$ because $c+5 > 2c+4$ and thus is out of range. Therefore, the Rado number is different from that in \cite{1916Schur}
\end{rem}
\begin{case}
    At least one among $c$ and $q$ is even and $c >1$.
\end{case}
To show that the $R_2(c,q)$ is \(r = (q+1)(c+2)+c+1\), we have to establish the following two statements :  
\begin{itemize}
    \item there exists a coloring of the integers from \(1\) to \(r-1\) which admits neither a red solution to $c$ equation nor a blue solution to $q$ equation, thereby yielding a lower bound.
    \item for any red--blue coloring of the integers from \(1\) to \(r\),  either a red solution to $c$ equation or a blue solution to $q$ equation must necessarily occur, which establishes the upper bound.
\end{itemize}
\begin{claim}\label{clm3.3}
    $R_2(c,q) \geq (q+1)(c+2) + c+1$.
\end{claim} 

 \begin{proof} Let $R =  \{1,2,\dots, c+1\}\ \cup\ \{(q+1)(c+2),\dots,(q+1)(c+2)+c\}$ and $B = \{c+2,\dots,q(c+2)+c+1\}$. Consider the $2$-coloring of $[1,(q+1)(c+2)+c]$ defined by,
 $$
\chi(n) = \text{Red }, \text{ if $n \in R$},
\qquad
\chi(n) = \text{Blue }, \text{ if $n \in B$}.
$$
We show that under this coloring, there is neither a red solution to the $c$ equation, nor a blue solution to the $q$ equation.

\smallskip
\noindent\textbf{No blue solution to $q$ equation:}\\
 If both $x$ and $y$ are blue, then $x + qy \geq q(c+2) + (c+2)$. Thus, either $z \in R$ and therefore $z$ is red, or $z > (q+1)(c+2)+c$. Therefore there is no blue solution. 

\smallskip
\noindent\textbf{No red solution to $c$ equation:}\\
If $x, y \in \{1,2,\dots,c+1\}$, then $$x+y+c \geq c+2\ \text{ and }  \quad x+y+c \leq 3c+2 < (q+1)(c+2)$$ Thus $z \in B$. 
 Therefore, without loss of generality, we can assume $y \in  \{(q+1)(c+2),\dots,(q+1)(c+2)+c\}$. Thus,
    $$x+y +c \geq 1 + (q+1)(c+2) + c $$
which is out of range. Therefore, no red solution to $x+y+c=z$ exists.
Therefore, we get $R_2(c,q) \geq (q+1)(c+2)+c+1$  \end{proof}

\begin{claim}\label{clm3.4}
    $R_2(c,q) \leq (q+1)(c+2) + c +1$
\end{claim}
\begin{proof} Assume for the sake of contradiction, there is a coloring of $\{1,2,\dots,(q+1)(c+2)+c+1\}$ such that no red solution to $x+y+c=z$ and no blue solution to $x+y=z$ exists, respectively. We derive a
contradiction by analyzing the color assigned to $1$. Recall that, under the current case, we have either $c$ or $q$ to be even and $q >1$.

\begin{subcase}
    $c$ is even.
\end{subcase}
\begin{subsubcase} $1$ is colored red.\end{subsubcase}
\begingroup
\setlength{\jot}{2pt} 
\begin{align*}
(1,1,c+2)_c&\Longrightarrow {c+2} \text{ is blue}.\\
(c+2,c+2,(q+1)(c+2))_q&\Longrightarrow {(q+1)(c+2)}\text{ is red}.\\
((q+1)(c+2),1,(q+1)(c+2)+c+1)_c
  &\Longrightarrow {(q+1)(c+2)+c+1}\text{ is blue}.\\
(2c+3,c+2,(q+1)(c+2)+c+1)_q&\Longrightarrow {2c+3}\text{ is red}.\\
\Bigl(\tfrac{q(c+2)}{2}+1,\tfrac{q(c+2)}{2}+1,(q+1)(c+2)\Bigr)_c
  &\Longrightarrow {\tfrac{q(c+2)}{2}+1} \text{ is blue}.\\
(q(c+2)+1,1,(q+1)(c+2))_c&\Longrightarrow {q(c+2)+1} \text{ is blue}.\\
\Bigl(\tfrac{q(c+2)}{2}+1,\tfrac{c+2}{2},q(c+2)+1\Bigr)_q
  &\Longrightarrow {\tfrac{c+2}{2}}\text{ is red}.\\
\Bigl(\tfrac{c+2}{2},\tfrac{c+2}{2},2c+2\Bigr)_c&\Longrightarrow {2c+2} \text{ is blue}.\\
(2c+2,c+2,q(c+2)+2c+2)_q&\Longrightarrow {q(c+2)+2c+2} \text{ is red}. \displaybreak[2]\\
(q(c+2)+c+1,1,q(c+2)+2c+2)_c&\Longrightarrow {q(c+2)+c+1}\text{ is blue}.\\
(c+1,c+2,q(c+2)+c+1)_q&\Longrightarrow {c+1}\text{ is red} .\\
(c+1,2,2c+3)_c&\Longrightarrow {2}\text{ is blue}.\\
(2,2,2q+2)_q&\Longrightarrow {2q+2}\text{ is red}.\\
(2c+2,2,2q+2+2c)_q&\Longrightarrow {2q+2+2c}\text{ is red}.\\
(2q+2,c,2q+2+2c)_c&\Longrightarrow {c}\text{ is blue}.\\
(c,c,qc+c)_q&\Longrightarrow {qc+c}\text{ is red}.
\end{align*}
\text{Thus }$(2q+2,\ qc+c,\ q(c+2)+2c+2)$\text{ is a red solution, a contradiction.}
\endgroup

\begin{subsubcase}
    $1$ is colored blue.
\end{subsubcase}
\begingroup
\setlength{\jot}{2pt}
\begin{align*}
(1,1,q+1)_q&\Longrightarrow q+1 \text{ is red.}\\
(q+1,q+1,2q+2+c)_c&\Longrightarrow 2q+2+c \text{ is blue.}\\
(q+c+2,1,2q+2+c)_q&\Longrightarrow q+c+2 \text{ is red.}\\
(q+c+2,q+1,2q+2c+3)_c&\Longrightarrow 2q+2c+3 \text{ is blue.}\\
(2q+2c+3,1,3q+2c+3)_q&\Longrightarrow 3q+2c+3 \text{ is red.}\\
(2q+2+c,1,3q+c+2)_q&\Longrightarrow 3q+c+2 \text{ is red.}\\
(q+1,2q+1,3q+c+2)_c&\Longrightarrow 2q+1 \text{ is blue.}\\
(1,2,2q+1)_q&\Longrightarrow 2 \text{ is red.}\\
(2,2,c+4)_c&\Longrightarrow c+4 \text{ is blue}\\
(c+4,1,c+q+4)_q&\Longrightarrow c+q+4 \text{ is red.}\\
(q+1,3,c+q+4)_c&\Longrightarrow 3 \text{ is blue.}\\
(3,3,3q+3)_q&\Longrightarrow 3q+3 \text{ is red.}\\
(3q+3,c,3q+2c+3)_c&\Longrightarrow c \text{ is blue.}\\
(c,c,qc+c)_q&\Longrightarrow qc+c \text{ is red.}\\
(\tfrac{qc}{2},\tfrac{qc}{2},cq+c)_c&\Longrightarrow \tfrac{qc}{2} \text{ is blue.}\\
(\tfrac{qc}{2},1,\tfrac{qc}{2}+q)_q&\Longrightarrow \tfrac{qc}{2}+q \text{ is red.}\\
(\tfrac{qc}{2}+q,\tfrac{qc}{2}+q,qc+2q+c)_c&\Longrightarrow qc+2q+c \text{ is blue.}\\
(2q+c,c,qc+2q+c)_q&\Longrightarrow 2q+c \text{ is red.}\\
(3,1,q+3)_q&\Longrightarrow q+3 \text{ is red.}
\end{align*}
\text{Thus }$(q+3,\ 2q+c,\ 3q+2c+3)$\text{ is a red solution, a contradiction.}
\endgroup

\begin{subcase}
    $c$ is odd, and therefore $q$ must be even.
\end{subcase}

\begin{subsubcase}
    $1$ is colored red.
\end{subsubcase}

\begin{align*}
(1,1,c+2)_c&\Longrightarrow c+2 \text{ is blue.}\\
(c+2,c+2,(q+1)(c+2))_q&\Longrightarrow (q+1)(c+2) \text{ is red.}\\
((q+1)(c+2),1, (q+1)(c+2)+c+1 )_c&\Longrightarrow (q+1)(c+2)+c+1 \text{ is blue.}\\
(2c+3,c+2,(q+1)(c+2)+c+1)_q&\Longrightarrow 2c+3 \text{ is red.}\\
\left(\tfrac{c+3}{2},\tfrac{c+3}{2},2c+3\right)_c&\Longrightarrow \tfrac{c+3}{2} \text{ is blue.}\\
(qc+2q+1,1,(q+1)(c+2))_c&\Longrightarrow qc+2q+1 \text{ is blue.}\\
\left(\tfrac{qc}{2}+q+1,\tfrac{qc}{2}+q+1,(q+1)(c+2)\right)_c&\Longrightarrow \tfrac{qc}{2}+q+1 \text{ is blue.}\\
(\tfrac{qc}{2}+\tfrac{q}{2}+1,\tfrac{c+3}{2},qc+2q+1)_q&\Longrightarrow \tfrac{qc}{2}+\tfrac{q}{2}+1 \text{ is red.}\\
(\tfrac{qc}{2}+\tfrac{q}{2}+1,1,\tfrac{qc}{2}+\tfrac{q}{2}+c+2)_c&\Longrightarrow \tfrac{qc}{2}+\tfrac{q}{2}+c+2 \text{ is blue.}\\
(c+2,\tfrac{c+1}{2},\tfrac{qc}{2}+\tfrac{q}{2}+c+2)_q&\Longrightarrow \tfrac{c+1}{2} \text{ is red.}\\
(1,\tfrac{c+1}{2},\tfrac{3(c+1)}{2})_c&\Longrightarrow \tfrac{3(c+1)}{2} \text{ is blue.}\\
(\tfrac{c+1}{2}+2,\tfrac{c+1}{2},2c+3)_c&\Longrightarrow \tfrac{c+1}{2}+2=\tfrac{c+5}{2} \text{ is blue.}\\
(\tfrac{c+1}{2},\tfrac{c+1}{2},2c+1)_c&\Longrightarrow 2c+1 \text{ is blue.}\\
(\tfrac{c+5}{2},\tfrac{c+5}{2},\tfrac{q(c+5)}{2}+\tfrac{c+5}{2})_q&\Longrightarrow \tfrac{q(c+5)}{2}+\tfrac{c+5}{2} \text{ is red.}\\
(\tfrac{3(c+1)}{2},\tfrac{c+5}{2},\tfrac{q(c+5)}{2}+\tfrac{3(c+1)}{2})_q&\Longrightarrow \tfrac{q(c+5)}{2}+\tfrac{3(c+1)}{2} \text{ is red.}\\
(2c+1,\tfrac{c+5}{2},\tfrac{q(c+5)}{2}+2c+1)_q&\Longrightarrow \tfrac{q(c+5)}{2}+2c+1 \text{ is red.}\\
(1,\tfrac{q(c+5)}{2}+c,\tfrac{q(c+5)}{2}+2c+1)_c&\Longrightarrow \tfrac{q(c+5)}{2}+c \text{ is blue.}\\
(c,\tfrac{(c+5)}{2},\tfrac{q(c+5)}{2}+c)_q&\Longrightarrow c \text{ is red.}\\
(3,c,2c+3)_c&\Longrightarrow 3 \text{ is blue.}\\
(3,\tfrac{(c+5)}{2},\tfrac{q(c+5)}{2}+3)_q&\Longrightarrow \tfrac{q(c+5)}{2}+3 \text{ is red.}\\
(1,\tfrac{(c+5)}{2},\tfrac{q(c+5)}{2}+c+4)_c&\Longrightarrow \tfrac{q(c+5)}{2}+c+4\text{ is blue.}\\
(c+4,\tfrac{(c+5)}{2},\tfrac{q(c+5)}{2}+c+4)_q&\Longrightarrow c+4\text{ is red.}\\
(2,2,c+4)_c&\Longrightarrow 2 \text{ is blue.}\\
(2,2,2q+2)_q&\Longrightarrow 2q+2 \text{ is red.}\\
(2,\tfrac{c+5}{2},\tfrac{q(c+5)}{2}+2)_q&\Longrightarrow \tfrac{q(c+5)}{2}+2 \text{ is red.}\\
(1,\tfrac{q(c+5)}{2}+2,\tfrac{q(c+5)}{2}+c+3)_c&\Longrightarrow \tfrac{q(c+5)}{2}+c+3 \text{ is blue.}\\
(c+3,\tfrac{(c+5)}{2},\tfrac{q(c+5)}{2}+c+3)_q&\Longrightarrow c+3 \text{ is red.}\\
(1,c+3,2c+4)_c&\Longrightarrow 2c+4 \text{ is blue.}\\
(2c+4,2,2c+4)_q&\Longrightarrow 2q+2c+4 \text{ is red.}\\
(c+4,2q,2q+2c+4)_c&\Longrightarrow 2q \text{ is blue.}\\
(2q,2,4q)_q&\Longrightarrow 4q \text{ is red.}\\
(4q,1,4q+c+1)_c&\Longrightarrow 4q+c+1 \text{ is blue.}\\
(q+c+1,3,4q+c+1)_q&\Longrightarrow q+c+1 \text{ is red.}\\
(1,q,q+c+1)_c&\Longrightarrow q \text{ is blue.}\\
(q,2,3q)_q&\Longrightarrow 3q \text{ is red.}\\
(3q,1,3q+c+1)_c&\Longrightarrow 3q+c+1 \text{ is blue.}\\
(c+1,3,3q+c+1)_q&\Longrightarrow c+1 \text{ is red.}\\
(2q+c+1,2,4q+c+1)_q&\Longrightarrow 2q+c+1 \text{ is red.}\\
(1,2q+c+1,2q+2c+2)_c&\Longrightarrow 2q+2c+2 \text{ is blue.}\\
(2c+2,2,2q+2c+2)_q&\Longrightarrow 2c+2 \text{ is red.}
\end{align*}
Thus, $(1,c+1,2c+2)$ is a red solution, a contradiction.

\begin{subsubcase}
    $1$ is colored blue.
\end{subsubcase}
\begingroup
\setlength{\jot}{2pt}
\begin{align*}
(1,1,q+1)_q&\Longrightarrow q+1 \text{ is red.}\\
(q+1,q+1,2q+2+c)_c&\Longrightarrow 2q+2+c \text{ is blue.}\\
(q+c+2,1,2q+2+c)_q&\Longrightarrow q+c+2 \text{ is red.}\\
(q+c+2,q+1,2q+2c+3)_c&\Longrightarrow 2q+2c+3 \text{ is blue.}\\
(2q+2c+3,1,3q+2c+3)_q&\Longrightarrow 3q+2c+3 \text{ is red.}\\
(2q+2+c,1,3q+c+2)_q&\Longrightarrow 3q+c+2 \text{ is red.}\\
(q+1,2q+1,3q+c+2)_c&\Longrightarrow 2q+1 \text{ is blue.}\\
(1,2,2q+1)_q&\Longrightarrow 2 \text{ is red.}\\
(2,2,c+4)_c&\Longrightarrow c+4 \text{ is blue}\\
(c+4,1,c+q+4)_q&\Longrightarrow c+q+4 \text{ is red.}\\
(q+1,3,c+q+4)_c&\Longrightarrow 3 \text{ is blue.}\\
(3,3,3q+3)_q&\Longrightarrow 3q+3 \text{ is red.}\\
(3q+3,c,3q+2c+3)_c&\Longrightarrow c \text{ is blue.}\\
(c,c,qc+c)_q&\Longrightarrow qc+c \text{ is red.}\\
(\tfrac{qc}{2},\tfrac{qc}{2},cq+c)_c&\Longrightarrow \tfrac{qc}{2} \text{ is blue.}\\
(\tfrac{qc}{2},1,\tfrac{qc}{2}+q)_q&\Longrightarrow \tfrac{qc}{2}+q \text{ is red.}\\
(\tfrac{qc}{2}+q,\tfrac{qc}{2}+q,qc+2q+c)_c&\Longrightarrow qc+2q+c \text{ is blue.}\\
(2q+c,c,qc+2q+c)_q&\Longrightarrow 2q+c \text{ is red.}\\
(3,1,q+3)_q&\Longrightarrow q+3 \text{ is red.}
\end{align*}
\text{Thus }$(q+3,\ 2q+c,\ 3q+2c+3)$\text{ is a red solution, a contradiction.}
\endgroup
Thus, we arrive at a contradiction in both the subcases. Therefore, $R_2(c,q) \leq (q+1)(c+2)+c+1$. \end{proof}

From Claim \ref{clm3.3} and \ref{clm3.4} we get that when either $c$ or $q$ is even, and $q >1$, $R_2(c,q) = (q+1)(c+2)+c+1$.
\begin{rem}
Note that the same proof does not follow when $c=0$ because our coloring starts from $1$ but we have used a coloring of $c$ in our computation. Therefore, the Rado number is different from that in \cite{myers2006two}.
\end{rem}

\vspace{2mm}
Thus, we get the two-color off-diagonal Rado number for all three cases. \end{proof}

\section{Conclusion}
The study of Rado numbers contains numerous open problems. It is well known that the complexity of these problems increases substantially with the number of colors involved. As a result, most existing research has concentrated on 2-color Rado numbers. In this context, several variants, including disjunctive  and off-diagonal Rado numbers, have been studied for various classes of linear or system of linear equations. While several variants of Rado numbers—including disjunctive \cite{johnson2005disjunctive} and off-diagonal \cite{jing2024some,myers_offdiagonal2006,robertson_schaal2001} forms—have been studied for different classes of equations, the off-diagonal literature has largely been confined to homogeneous systems. 

In this work, we evaluated the $2$-color off-diagonal rado number, $R_2(c,q)$, for $x+y+c=z$ and  $x+qy=z$. Future work may involve studying broader families of equations which can be a combination of homogeneous and heterogeneous equations and exploring possible generalizations. Also, known results in such framework can be extended to more than two colors.
\bibliographystyle{plain}
\bibliography{rado}
\end{document}